\documentclass[useAMS]{gDEA2e}

\begin{document}

%\doi{10.1080/1023619YYxxxxxxxx}
% \issn{1563-5120}
%\issnp{1023-6198} \jvol{00} \jnum{00} \jyear{2010} \jmonth{January}

\markboth{Berry and Sauer}{Journal of Difference Equations and Applications}

%\title{{\itshape Periodicity in sup-Contractive Non-Autonomous Recurrence Equations}}
\title{{\itshape Convergence of Periodically-Forced Rank-Type Equations}}

\author{Tyrus Berry\thanks{Corresponding author. Email: tberry@gmu.edu} and Timothy Sauer \\\vspace{6pt}  {\em{George Mason University; Fairfax, VA  22030}; }}

\newcommand{\ov}[0]{\overline}

\maketitle

\begin{abstract} Consider a difference equation which takes the k-th largest output of $m$ functions  of the previous $m$ terms of the sequence.  If the functions are also allowed to change periodically as the difference equation evolves this is analogous to a differential equation with periodic forcing.  A large class of such non-autonomous difference equations are shown to converge to a periodic limit which is independent of the initial condition.  The period of the limit does not depend on how far back each term is allowed to look back in the sequence, and is in fact equal to the period of the forcing.\end{abstract}

{\em To appear in the Journal of Difference Equations and Applications}

\section{Introduction}
Recently, there has been substantial interest in max-type difference equations \cite{GL,Voulov1,Voulov2,Voulov3,BFS,KR,Yang,BGKL,Chen}. 
For the  equation
\begin{equation} x_n = \max_{1\leq i \leq M} \{f_i(x_{n-i}) \} \end{equation}
with initial sequence $(x_1,...,x_M)$, it has been shown  \cite{sauerMax} that if the $f_i$ are contractive, then all solutions converge to a fixed point, i.e.
\begin{equation*} \lim_{n\to\infty} x_n = r_*. \end{equation*}
Moreover, the fixed point $r_*$ can be identified as the maximum of the individual fixed points 
\begin{equation*} r_* = \max_{1\le i\le M}\{r_i\}, \end{equation*}
where $r_i$ is the unique fixed point $r_i=f_i(r_i)$.

In view of this result, it is reasonable to ask about convergence in the periodically-forced case. Assume  that  $f_1, \ldots, f_M$ vary periodically with the discrete time variable $n$. We will investigate limiting behavior of the nonautonomous difference equation
\begin{equation}\label{maxeq} x_n = \max_{1\leq i \leq M} \{f_i(x_{n-i},n) \} \end{equation}
where $f_i:\mathbb{R}\times \mathbb{N} \to \mathbb{R}$ are contractions with respect to the first variable and $P$-periodic with respect to the second.  That is, we assume there exists $\alpha <1$ such that
\begin{equation} \label{alpha}  |f_i(x,n)-f_i(y,n)| \leq \alpha|x-y| 
\end{equation}
for all $i,n$, and 
$ f_i(x,n+P) = f_i(x,n)$ for all $n$.
We think of $P$ as the {\em forcing period} and $M$ as the {\em memory length}.  We will show below that solutions of the contractive, $P$-periodic difference equation (\ref{maxeq}) converge to a unique periodic orbit with period $P$, for all initial sequences.

In \cite{sauerRank}, rank-type equations were proposed as a generalization of max-type equations. Consider the difference equation
\begin{equation}\label{rankeq} x_n = \underset{1\leq i \leq M}{\textup{$k$-rank}} \{f_i(x_{n-i}) \} \end{equation}
where $k$-rank denotes the $k$th largest value. It was shown in \cite{sauerRank} that with the same contractiveness hypotheses, solutions converge to $r_* = \text{ $k$-rank\  } \{ r_i\}$, the $k$th largest of the individual fixed points. In this paper, we further generalize this result to the periodically-forced case. Our main result is the following.

\begin{theorem} \label{theorem1}
Let $f_1,\ldots,f_M:\mathbb{R}\times \mathbb{N} \to \mathbb{R}$ be contractions with respect to the first variable and $P$-periodic with respect to the second. Let $1\leq k \leq M$. Then for any initial sequence, the solution of the difference equation
\begin{equation}\label{rankeq1} x_n =  \underset{1\leq i \leq M}{\textup{$k$-rank}} \{f_i(x_{n-i},n) \} \end{equation}
is asymptotically periodic with period $P$. 
\end{theorem}

Unlike the autonomous case, in general we do not know how to find a formula for the periodic solution in terms of the individual dynamics of the $f_i$. In the fourth section of this article we relate some partial progress in this direction.

The conclusion of Theorem \ref{theorem1} fails if $\alpha\geq 1$ in (\ref{alpha}), even if $k=P=1$, the autonomous max-type case \cite{GL}. We mention two well-known examples where the condition on $\alpha$ does not hold.
\begin{example}
Consider the difference equation
\[x_n=\max\{-x_{n-1},-x_{n-2}\}.\]
Although this is an autonomous max-type equation ($k=P=1$), is easy to check that all solutions have period $3\neq P$. In this example, $\alpha=1$, and Theorem \ref{theorem1} does not apply.
\end{example}
\begin{example}
The first-order difference equation
\[x_n=\max\{1-2x_{n-1},2x_{n-1}-1\}\]
has bounded, nonperiodic (chaotic) solutions for almost every initial condition $x_0$. In this example, $\alpha=2$. As a dynamical system, this example is the upside-down tent map.
\end{example}
\begin{example}\label{examp0}
For a straightforward application of Theorem \ref{theorem1}, fix positive integers $P$ and $k\leq M$, and denote by $A$ and $B$ two $P\times M$ matrices of real numbers, where the entries $|A_{ij}|<1$. Define the difference equation
\begin{equation}
x_n = \textup{$k$-rank} \{A_{\overline{n}1}x_{n-1}+B_{\overline{n}1}, A_{\overline{n}2}x_{n-2}+B_{\overline{n}2}, \ldots, A_{\overline{n}M}x_{n-M}+B_{\overline{n}M}\}
\end{equation}
where $\overline{n}\equiv  1+(n-1 \mod P)$. Thus the coefficients cycle through the rows of the matrices $A$ and $B$, forcing the equation with period $P$. Theorem \ref{theorem1} implies that all solutions converge asymptotically to a period $P$ solution.
\end{example}

\begin{example}\label{ex1} Let $A$ be a $P\times M$ matrix of positive real numbers and consider the recurrence
\begin{equation} \label{eq31}
 x_n = \max\{A_{\ov n1}x_{n-1}^{\alpha_1},\ldots,A_{\ov nM}x_{n-M}^{\alpha_m} \}
 \end{equation}
where $-1 < \alpha_i < 1$ and $\overline{n}= 1+(n-1 \mod P)$.
It follows from Theorem \ref{theorem1} (applied to $y_n=\ln x_n$) that the recurrence converges to a unique period $P$ orbit.   Note that this result is independent of the memory length $M$. The periodicity of the limit only depends on the periodicity of the forcing. 

In sufficiently simple cases, we can say more about the convergent solution.  In Section 4, we further pursue the special case $M=P=2$ of (\ref{eq31}). A closed form for the globally attracting solution of (\ref{eq31})  is given by
\begin{eqnarray*} \lim_{n\to\infty} x_{2n} &=& \max \left\{ A_{12}A_{21}^{\frac{\alpha_1}{1-\alpha_2}},A_{11}^{\frac{\alpha_1}{1-\alpha_1^2}}A_{12}^{\frac{1}{1-\alpha_1^2}}, A_{22}^{\frac{1}{1-\alpha_2}} \right\} \\
\lim_{n\to\infty} x_{2n+1} &=& \max \left\{ A_{11}A_{22}^{\frac{\alpha_1}{1-\alpha_2}},A_{12}^{\frac{\alpha_1}{1-\alpha_1^2}}A_{11}^{\frac{1}{1-\alpha_1^2}}, A_{21}^{\frac{1}{1-\alpha_2}} \right\}. 
 \end{eqnarray*}
 (See Example \ref{ex44} for details.)
The complexity of this solution contrasts with the simplicity of the case $P=1$, which is simply $x_n \to \max\left\{A_i^{\frac{1}{1-\alpha_i}}\right\}$ (see \cite{sauerMax,Sun,Stev}).
\end{example}
%Another interesting effect involves the forcing term $A(n) = A\sin(2\pi n/P)$.  This term is periodic for $P\in \mathbb{Z}$ but if we perturb $P$ slightly from integer value then $A(n)$ becomes quasi-periodic as in the next example:
\begin{example} Let $A_1, A_2, A_3$ be real numbers less than $0.15$, $B_1,B_2, B_3$ be arbitrary real numbers, and $P$ be a positive integer. Then it follows from Theorem \ref{theorem1} that every solution of
\[ x_n = \textup{median} \left\{ e^{A_1\sin(B_1+2\pi n/P) - x_{n-1}^2}, e^{A_2\sin(B_2+2\pi n/P) - x_{n-2}^2},
e^{A_3\sin(B_3+2\pi n/P) - x_{n-3}^2}\right\} \]
is asymptotically periodic with period $P$.  Note that median is synonymous with 2-rank. It is easily checked that the condition $A_i<0.15< \frac{1}{2} - \frac{\ln 2}{2}$ implies that the functions $f_i(x)=\exp(A_i\sin(B_i+2\pi n/P)-x^2)$ are contractive, so the convergence is implied by Theorem \ref{theorem1}.
\end{example}
%To see this consider the function $f(x) = e^{A-x^2}$ then $f'(x) = -2xe^{A-x^2}$.  This derivative is maximized when $x = 2^{-1/2}$ and so $|f'(x)| \leq 2\sqrt{2}e^{A-1/2}$.  For $f$ to be a contraction requires that $|f'(x)| < 1$, which is equivalent to:
%\[ 2\sqrt{2}e^{A-1/2} < 1 \hspace{30pt} \Leftrightarrow \hspace{30pt} A < \frac{1}{2} - \frac{3\ln 2}{2} \]
%Thus if $A_i < -\frac{1}{2} - \frac{3\ln 2}{2}$ then $A_i + \sin(2\pi n/P) <  \frac{1}{2} - \frac{3\ln 2}{2}$ which makes all the functions inside the k-rank contractions.  Unfortunately the result does not apply to general $P \in \mathbb{R}$ since the forcing is no longer periodic.  Note that the solution to this example will not have a simple expression since the component fixed points are the solution to $x = e^{A-x^2}$.

Our main convergence result Theorem \ref{theorem1} will be proved as a special case of a more general result, Corollary \ref{mainresult}, which applies to a class of difference equations called sup-contractive. The next example is covered under the sup-contractive hypothesis, which is more general than a fixed $k$-rank.

\begin{example}\label{ex3} Let $f_1,...,f_M : \mathbb{R} \to \mathbb{R}$ be contractions.  Let $P \leq M$ be a positive integer and denote  $\ov{n}=1+(n\mod P)$.  We will show in Example \ref{ex37} that every solution of
\[ x_n = \frac{1}{2} \left( \max_{1\leq i\leq M}\{f_i(x_{n-i})\} - \underset{1\leq i\leq M}{\overline{n}\textup{-rank}}
\{f_i(x_{n-i})\} \right) \]
is asymptotically periodic with period $P$.  
\end{example}

In Section 2 below, we develop some facts about functions that are contractive in the sup-norm.  An important result is Lemma \ref{ranknonexp} which will imply that (\ref{rankeq1}) is sup-contractive.  In Section 3, we show that the general class of sup-contractive recurrence equations converges as desired.  Finally, in Section 4, we return to the max-type equation (\ref{maxeq}), and find the value of the asymptotically convergent periodic orbit for some specific values of M and P.

\begin{section}{Sup-Contractive Functions}
The convergence proofs in the next section apply to a wide class of functions $G: \mathbb{R}^a \to \mathbb{R}^b$, that includes compositions of contractive functions as in (\ref{alpha}) with the max and $k$-rank functions. We will be interested in the operator norm of functions $G$, where the norm used is  the sup norm. Namely, assume there exists $L \geq 0$ such that 
\[ ||G(x) - G( y)||_{\infty} \leq L ||x - y||_{\infty} \]
for all $ x, y \in \mathbb{R}^a$. We can think of $L$ as the Lipschitz constant of $G$ in the sup norm, and will refer to it as the {\it sup-Lipschitz constant} in the following sections. When $L = 1$ we will call $G$ {\it sup-non-expansive} and when $L<1$ we will call $G$ {\it sup-contractive}.  Note that a contraction on $\mathbb{R}$ is sup-contractive since the infinity norm on $\mathbb{R}$ is simply the absolute value.

%
%In this article we will treat recurrence equations of the form (\ref{rankeq1}) in a slightly more general context, using a sup-non-expansive function $G$ as:
%\[ x_n =  G(f_1(x_{n-1},n), f_2(x_{n-2},n),...,f_M(x_{n-M},n)) \]
%Thus we collect three interesting examples of scalar-valued sup-Lipschitz functions:
\begin{example} Let $G:\mathbb R^M\to \mathbb R$ be the function $G(x) = c||x||_p$ where $c >0$. Since
\[ |c ||x||_p - c ||y||_p| = c| ||x||_p - ||y||_p | \leq c||x-y||_p \leq c M^{1/p} ||x-y||_{\infty}, \]
$G$ is sup-non-expansive when  $c = M^{-1/p}$ and sup-contractive when  $c < M^{-1/p}$.
\end{example}

\begin{example} For a fixed $z \in \mathbb{R}^M$, define $G:\mathbb R^M\to \mathbb R$ by $G(x) = z^T x$. Since
\[ | z^T x - z^T y| = | z^T( x- y)| = \sum_{i=1}^M |z_i||x_i-y_i| \leq || x-  y||_{\infty} \sum_{i=1}^M |z_i| = || x -  y||_{\infty} || z||_1, \]
$G$ is sup-non-expansive when $|| z||_1= 1$, and sup-contractive when $|| z||_1 < 1$.
\end{example}

\begin{example} For $1\leq k\leq M$, let $R_k : \mathbb{R}^M \to \mathbb{R}$ be the function that returns the $k$-th largest entry of the $M$-dimensional input vector.  Customarily, $R_k$ is called the ``$k$-rank" function. It includes the $\max$ $(k=1)$ and $\min$ $(k=M)$  as special cases. Not surprisingly, the max function is sup-non-expansive. Somewhat more surprising is that this property holds for all $k$, as shown in the next Lemma.
\begin{lemma} \label{ranknonexp} The $k$-rank function $R_k$ is sup-non-expansive.
\end{lemma}
\begin{proof} Let $ x, y \in \mathbb{R}^M$, and assume, without loss of generality, that $R_k( y) \leq R_k( x)$.  Set $y_{r(i)}$ be the $i$-th largest component of $ y$ and let $x_{s(i)}$ be the $i$-th largest component of $ x$.  Then:
\[ y_{r(1)} \leq y_{r(2)} \leq \cdots \leq y_{r(k)} = R_k( y) \leq R_k( x) = x_{s(k)} \leq x_{s(k+1)} \leq \cdots \leq x_{s(M)} \]
Now examine the list of natural numbers $I = (r(1),...,r(k),s(k),...,s(M))$.  We note that $I$ has length $M+1$ but each entry is chosen from $\{1,...,M\}$.  Thus by the pigeonhole principle at least two of the listed numbers must be the same.  Note that all the $r(1),...,r(k)$ are distinct and all the $s(k),...,s(M)$ are distinct, thus there must be some $1\leq i \leq k$ and some $k \leq j \leq M$ such that $r(i) = s(j) = t$.  Therefore
\[ y_t = y_{r(i)} \leq R_k( y) \leq R_k( x) \leq x_{s(j)} = x_t, \]
which implies that
\[ |R_k( x) - R_k( y)| = R_k( x) - R_k( y) \leq x_t - y_t = |x_t - y_t| \leq  ||x- y||_{\infty} \]
Since the absolute value is the infinity norm on $\mathbb{R}$,
 $R_k$ is sup-non-expansive.
 \end{proof}
\end{example}

The sup-Lipschitz constants are multiplied under composition, based on the next Lemma.
\begin{lemma}\label{supcompositions} Assume $f_1,...,f_k: \mathbb{R}^M \to \mathbb{R}$ have sup-Lipschitz constants less than $L_1$ and $f: \mathbb{R}^k \to \mathbb{R}$ has sup-Lipschitz constant $L_2$, then $g( x) = f(f_1( x),f_2( x),...,f_k( x))$ has sup-Lipschitz  constant no larger than $L_1L_2$.
\end{lemma}
\begin{proof} Let $ x,  y \in \mathbb{R}^M$ then:
\begin{eqnarray*} |g( x) - g( y)| &=& | f(f_1( x),...,f_k( x))- f(f_1( y),...,f_k( y))| \\ 
 &\leq& L_2 \max_{1\leq i \leq k} |f_i( x) - f_i( y)| \\ &\leq& L_1L_2 ||x-y||_{\infty} 
\end{eqnarray*}
\end{proof}
Note that if $f$ is sup-contractive, and all $f_i$ are sup-non-expansive then $g$ is sup-contractive.  Similarly if $f$ is sup-non-expansive, and all $f_i$ are sup-contractive then $g$ is sup-contractive.  Finally, the next Lemma shows that if we bundle together functions into a vector, the sup-Lipschitz constant cannot grow.
\begin{lemma} \label{lem26} Assume $f_1,...,f_k:\mathbb{R}^M \to \mathbb{R}$ have sup-Lipschitz constants at most $L$.  Then $f:\mathbb{R}^M\to \mathbb{R}^k$ defined by 
\[ f( x) = (f_1( x),f_2( x),...,f_k( x))\] has sup-Lipschitz constant at most $L$.
\end{lemma}
\begin{proof} Let $ x, y \in \mathbb{R}^M$ then:
\[ || f( x) -  f( y)||_{\infty} = \max_{1 \leq i \leq k} \{ |f_i( x) - f_i( y)|\} \leq L || x -  y||_{\infty} \]
\end{proof} 
Note that if $k=M$ and $L<1$ then in fact $ f$ is a contraction on $\mathbb{R}^M$, so bundling  functions is a way to build contractions.  Finally, note that these statements could be generalized to allow $f_i$ to be vector valued but this is not necessary for what follows.
\end{section}

\begin{section}{General Case}
We now return to our original recurrence equation
\[ x_n = \underset{1\leq i \leq M}{\textup{$k$-rank}} \{f_i(x_{n-i},n) \}\]
under the assumption that each $f_i:\mathbb{R}\to \mathbb{R}$ is a contraction.
Since $\textup{$k$-rank}$ is sup-non-expansive by Lemma \ref{ranknonexp}, and each $f_i$ is sup-contractive, the composition of the two is sup-contractive by Lemma \ref{supcompositions}.  This is a special case of the equation
\[ x_n = G_{n}(x_{n-1},...,x_{n-M}) \]
where $G_1,...G_{P}$ are any sup-contractive functions and $G_{n+P} = G_n$ for all $n$.  Setting $s = PM$, we will first show that for all initial conditions the limit is periodic of period $s$ by a contraction mapping argument.  We will then use the $P$-periodicity of $G_n$ to show that the limit is actually periodic of period $P$.

With slight abuse of notation we can consider $G_n$ as a function of all $s$ variables, although it depends only on the first M:
\begin{equation}\label{Gdef} x_n = G_{n}(x_{n-1},...,x_{n-s})= G_{n}(x_{n-1},...,x_{n-M}) \end{equation}
where $G_{n+s} = G_{n+MP} = G_n$ for all $n$.  We define 
\[ \ov x_n = (x_{ns+1},...,x_{ns+s}) \]
and let $\ov x_0$ be the initial condition.  We want to show that we can write 
\[ \ov x_{n+1} = \ov F(\ov x_n) \]
where $\ov F: \mathbb{R}^s \to \mathbb{R}^s$ is a contraction in the infinity norm on $\mathbb{R}^s$.  Let $\ov y = (y_1,...,y_s)$ and define the following functions for $k = 1,...,s$.
\begin{eqnarray*} F_1(\ov y) &=& G_1(y_{s},y_{s-1},...,y_{1}) \\
 F_k(\ov y) &=& G_k(F_{k-1}(\ov y),...,F_{1}(\ov y),y_s,y_{s-1},...,y_k) \end{eqnarray*}
Note that this is an inductive definition since $F_k$ depends on $F_1,...,F_{k-1}$.  Finally, we define 
\begin{equation}\label{Fdef} \ov F(\ov y) = (F_1(\ov y),...,F_s(\ov y))\end{equation}
The recursive nature of the definition of $\ov F$ requires the following Lemma to show that $\ov{F}$ encapsulates the evolution of the sequence $\{x_n\}$.
\begin{lemma} Let $G_1,...,G_s:\mathbb{R}^s \to \mathbb{R}$ and let $\{x_i\}$ be defined by \eqref{Gdef}.  Define $\overline{F}$ by \eqref{Fdef}.  Then for all $n$ we have $\ov x_{n+1} = \ov F(\ov x_n)$. \end{lemma}
\begin{proof} This is equivalent to showing that $x_{ns+s+k} = F_k(\ov x_n)$ for $k = 1,...,s$ which is equivalent to:
\[ F_k(\ov x_n) = G_k(x_{ns+s+k-1},...,x_{ns+k}) \] 
Note that when $k=1$ we have:
\[ F_1(\ov x_n) = G_1(x_{ns+s},...,x_{ns+1}) = x_{ns+s+1} \]
By definition of $F_1$ and $G_1$.  Then for $2 \leq k \leq s$ we can proceed by induction.  Assuming $x_{ns+s+\overline{k}} = F_{\ov k}(\ov x_n)$ for all $\ov k < k$ we have:
\begin{eqnarray*} F_k(\ov x_n) &=& G_k( F_{k-1}(\ov x_n),...,F_{1}(\ov x_n) ,x_{ns+s},...,x_{ns+k} ) \\
 &=& G_k(x_{ns+s+k-1},...,x_{ns+s+1} ,x_{ns+s},...,x_{ns+k} ) =  x_{ns+s+k} 
 \end{eqnarray*}
Which completes the proof.
\end{proof}
Now that we have defined $\ov F$ we need to show that it is a contraction.  
\begin{theorem}\label{contractiontheorem} Let $G_1,...,G_s:\mathbb{R}^s \to \mathbb{R}$ be sup-contractive and define $\ov{F}$ as in \eqref{Fdef}.  Then $\ov F: \mathbb{R}^s \to \mathbb{R}^s$ is a contraction with respect to the infinity norm.
\end{theorem}
\begin{proof} By Lemma \ref{lem26} it suffices to show that each $F_i$ is sup-contractive.  Note that the projection function $\pi_i(\ov y) = y_i$ is sup-non-expansive since:
\[ |\pi_i(\ov x) - \pi_i(\ov y)| = |x_i - y_i| \leq ||\ov x - \ov y||_{\infty} \]
Thus $F_1 = G_1(\pi_s(\ov y),...,\pi_1(\ov y))$ is sup-contractive by Lemma \ref{supcompositions} because $G_1$ is sup-contractive and $\pi_i$ is sup-non-expansive.  For $1<i\leq s$ we proceed by induction.  Assume $F_j$ is sup-contractive for all $j < i$, then: 
\[ F_i(\ov y) = G_i(F_{i-1}(\ov y),...,F_{1}(\ov y),\pi_s(\ov y), \pi_{s-1}(\ov y),...,\pi_i(\ov y)) \]
So $F_i$ is a composition of the sup-contractive function $G_i$ with sup-contractive functions $F_{i-1},...,F_1$ and sup-non-expansive projection functions $\pi_s,...,\pi_i$.  Thus by Lemma \ref{supcompositions}, $F_i$ is sup-contractive, which completes the induction and thus shows that $\ov F$ is a contraction.
\end{proof}
\begin{corollary}\label{PMperiodicity}  Let $G_1,...,G_P:\mathbb{R}^M \to \mathbb{R}$ be sup-contractive and let $G_{n+P} = G_P$ for all $n$.  Given an initial condition $(x_1,...,x_M)$ let
\[ x_n = G_n(x_{n-1},...,x_{n-M}) \]
then $x_n$ converges to a unique $PM$-periodic orbit independent of initial conditions. 
\end{corollary}
\begin{proof} By Theorem \ref{contractiontheorem} we can construct a contraction mapping $\ov F: \mathbb{R}^s \to \mathbb{R}^s$ such that $\ov x_{n+1} = \ov F(\ov x_n)$.  Thus, by the Contraction Mapping Theorem, $\ov F$ has a unique fixed point $x^* \in \mathbb{R}^s$ and for any $\ov x_0 \in \mathbb{R}^s$ we have $\lim_{n\to\infty}\ov x_n = x^*$.
\end{proof}
Note that $s$ may not be the prime period (smallest possible period).  So we now show that in fact $x_n$ must be asymptotically periodic of period $P$.  
\begin{theorem}\label{Pperiodicity} Let $G_1,...,G_P:\mathbb{R}^M \to \mathbb{R}$ be sup-contractive and define $\ov{F}$ as in \eqref{Fdef}.  Let $x^*$ be the unique fixed point of $\ov F$.  Then $ x^*$ is periodic of period $P$.
\end{theorem}
\begin{proof} Define a shift operator by $S(\ov x_n) = (x_{ns+1+P},...,x_{ns+s+P})$.  We will use the fact that $G_i$ is actually periodic of period $P$ to show that the function $S$ commutes with $\ov F$.  Note that
\begin{eqnarray*} S(\ov F(\ov x_n)) &=& S(\ov x_{n+1}) = (x_{ns+s+1+P},...,x_{ns+s+s+P})  \\
  \ov F(S(\ov x_n)) &=& \ov F(x_{ns+1+P},...,x_{ns+s+P}) 
\end{eqnarray*}
First we examine the first component of $\ov F(S(\ov x_n))$:
\begin{eqnarray*} ( \ov F(S(\ov x_n)))_1 &=& G_1(x_{ns+s+P},...,x_{ns+1+P}) \\
 &=& G_{1+P}(x_{ns+s+P},...,x_{ns+1+P}) = x_{ns+s+1+P} 
 \end{eqnarray*}
where the second equality comes from the fact that $G_i$ is periodic of period P.  This shows that the first components of $S(\ov F(\ov x_n))$ and $\ov F(S(\ov x_n))$ are the same.  We proceed inductively to show that all the components are the same.  Assume that $(S(\ov F(\ov x_n)))_j = (\ov F(S(\ov x_n)))_j$ for all $j<i$.  Then:
\begin{eqnarray*} ( \ov F(S(\ov x_n)))_i &=& G_i(F_{i-1}(S(\ov x_n)),...,F_1(S(\ov x_n)),x_{ns+s+P},...,x_{ns+s+P + i-s}) \\
&=& G_{i+P}(x_{ns+s+i-1+P},...,x_{ns+s+1+P},x_{ns+s+P},...,x_{ns+s+P + i-s}) \\
&=& x_{ns+s+i+P} 
\end{eqnarray*}
where we have again used the P-periodicity if $G_i$ to conclude $G_i = G_{i+P}$.  This shows that
\[ S(\ov F(\ov x_n)) =  \ov F(S(\ov x_n)) \]
So inductively we have $S(\ov x_n) = \ov F^n(S(\ov x_0))$.  Let $x^*$ be the unique fixed point of $\ov F$ and define two sequences, the first with $x_0 = x^*$ and the second with $y_0 = S(x^*)$.  Note that
\[  \lim_{n\to\infty} y_n = \lim_{n\to\infty} \ov F^n(\ov y_0) = x^* \]
since all initial conditions converge to $x^*$, and at the same time:
\[  \lim_{n\to\infty} y_n = \lim_{n\to\infty} \ov F^n(\ov y_0) = \lim_{n\to\infty} \ov F^n(S(x_0)) = \lim_{n\to\infty} \ov S(\ov F_n(x_0))= \lim_{n\to\infty} S(x_0) = S(x^*) \]
So we conclude that $S(x^*) = x^*$ and thus $x^*$ is periodic of period $P$. 
\end{proof}
\begin{corollary}\label{mainresult} Let $G_1,...,G_P:\mathbb{R}^M \to \mathbb{R}$ be sup-contractive and let $G_{n+P} = G_P$ for all $n$.  Given an initial condition $(x_1,...,x_M)$ let
\[ x_n = G_n(x_{n-1},...,x_{n-M}) \]
then $x_n$ converges to a unique $P$-periodic orbit independent of initial conditions. 
\end{corollary}
\begin{proof} By Theorem \ref{Pperiodicity} there exists a unique $x^* \in \mathbb{R}^s$ which is $P$-periodic such that $\lim_{n\to\infty} \ov x_n = x^*$.  Thus $x_n$ is asymptotically periodic of period P.
\end{proof} 
We conclude that $x_n$ approaches a unique periodic orbit, for any initial condition, whose period is equal to the forcing period $P$.  The periodicity of the rank-type equation (\ref{rankeq1}) is now an easy Corollary.
\begin{corollary}\label{rankResult} For $i=1,...,M$ let $f_i(x,n):\mathbb{R}\times\mathbb{N}\to\mathbb{R}$ be contractive in $x$ and $P$-periodic in $n$.  Given an initial condition $(x_1,...,x_M)$ and $k \in \{1,...,M\}$ let
\[ x_n = \textup{k-rank}\{f_i(x_{n-i},n)\} \]
then $x_n$ converges to a unique $P$-periodic orbit independent of initial conditions.
\end{corollary}
\begin{proof} Recall that by Lemma \ref{ranknonexp}, the $\textup{k-rank}$ function is sup-non-expansive, and each $f_i(x,n)$ is sup-contractive in $x$ so by Lemma \ref{supcompositions}, the composition is sup-contractive.  By Corollary \ref{mainresult}, $x_n$ is asymptotically periodic of period P.
\end{proof}
We can now return to the equation from Example \ref{ex3}.
\begin{example} \label{ex37}
Let $f_1,...,f_M : \mathbb{R} \to \mathbb{R}$ be contractions.  Let $P \leq M$ be a positive integer and set $\ov{n}=1+(n\mod P)$.  Let
\[ x_n = G_n(x_{n-1},...,x_{n-M}) = \frac{1}{2} \left( \max\{f_i(x_{n-i})\} - \overline{n}\textup{-rank}\{f_i(x_{n-i})\} \right) \]
Recall that by Lemma \ref{ranknonexp} the rank functions are all sup-non-expansive.  Thus $\ov n\textup{-rank}\{f_i(x_{n-i})\}$ is a composition of a sup-non-expansive function with the sup-contractive functions $f_i$, and thus the composition is sup-contractive by Lemma \ref{supcompositions}.  Furthermore, setting $\ov z = (1/2,-1/2)$, we see that $||\ov z||_1 = 1$ so $f(x) = z^T x$ is sup-non-expansive. Therefore
\begin{eqnarray*} G_n(y_1,...,y_M) &=& z^T \left(\max\{f_i(y_{i})\},\overline{n}\textup{-rank}\{f_i(y_{i})\right) \\
&=& \frac{1}{2} \left( \max\{f_i(y_{i})\} - \overline{n}\textup{-rank}\{f_i(y_{i})\} \right)
\end{eqnarray*}
is sup-contractive for all $n$.  By Corollary \ref{mainresult}, $x_n$ is asymptotically periodic with period $P$.
\end{example}
\end{section}

\begin{section}{Finding the Periodic Limit}
We now return to the max-type equation \eqref{maxeq} and rank-type equation \eqref{rankeq1} and attempt to find closed-form solutions.  The closed-form solution to the autonomous contractive rank-type equation was first given in \cite{sauerRank}, and we are able to reprove this result as a special case.  However, we will see that finding a closed formula for the limit under periodic forcing is in general more difficult.   
We will need two lemmas about contractions on $\mathbb{R}$.
\begin{lemma} \label{lem1} Let $f: \mathbb{R} \to \mathbb{R}$ be a contraction with fixed point $r$ then $x > r$ implies $f(x) < x$, and $x < r$ implies $f(x) > x$.
\end{lemma}
\begin{proof} 
Since $f(r)=r$ and $f$ is a contraction, $|f(x)-r|<|x-r|$. If $x>r$, then
\[f(x)-r\leq |f(x)-r|<|x-r|=x-r,\]
so $f(x)<x$. If $x<r$, then
\[r-f(x)\leq |f(x)-r|<|x-r|=r-x,\]
so $f(x)>x$.
\end{proof}
\begin{lemma} \label{lem21} Let $f,g: \mathbb{R} \to \mathbb{R}$ be contractions, and assume that $r_1, r_2, r_3, r_4$ satisfy $f(r_2)=r_1<r_3$ and $g(r_1)=r_2<r_4$. Then either $f(r_4)<r_3$ or $g(r_3)<r_4$.
\end{lemma}
\begin{proof} 
Assume $g(r_3)\geq r_4$. Then
\[ r_4-r_2=r_4-g(r_1)\leq g(r_3)-g(r_1) < r_3-r_1.\]
The contractivity of $f$ yields
\[f(r_4)-r_1\leq |f(r_4)-r_1|=|f(r_4)-f(r_2)|< |r_4-r_2|=r_4-r_2< r_3-r_1\]
which implies that $f(r_4)<r_3$.
\end{proof}

First, we can use Lemma \ref{lem1} to find the explicit solution of the autonomous equation \eqref{rankeq} where each $f_i$ is a contraction with fixed point $r_i$.  The following result, first proved in \cite{sauerRank}, shows that every initial condition converges to the constant solution $\textup{$k$-rank}\{r_i\}$.  
\begin{theorem}\label{periodone} 
 Let $f_i:\mathbb{R}\to\mathbb{R}$ be a contraction with fixed point $r_i$ for $1\leq i \leq M$ and set $ x_{n} = \textup{$k$-rank} \{ f_i(x_{n-i})\}$.  Then $\lim_{n\to\infty} x_n = \textup{k-rank} \{r_i\}$.
\end{theorem}
\begin{proof} Note that this recurrence is autonomous and thus $P=1$ so by Corollary \ref{rankResult} every initial condition converges to a unique constant (period one) solution $x^*$.  Thus it remains only to show that $x^* = \textup{$k$-rank}\{r_i\}$ is a fixed point of the recurrence.  Let $\sigma: \{1,...,M\} \to \{1,...,M\}$ be a permutation such that:
\[ r_{\sigma(1)} \leq r_{\sigma(2)} \leq \cdots \leq r_{\sigma(M)} \]
So the constant solution should be $x^* = r_{\sigma(k)}$.  Assume $x_1 = x_2 = \cdots = x_M = x^*$.  Then: 
\begin{eqnarray*} x_{M+1} &=& \textup{$k$-rank}\{f_1(x^*),...,f_{\sigma(k)}(x^*),...,f_M(x^*)\}  \\
&=& \textup{$k$-rank}\{f_1(x^*),...,x^*,...,f_M(x^*)\} 
\end{eqnarray*}
We want to show that $x_{M+1} = x^*$.  Note that when $i < k$ we have $r_{\sigma(i)} \leq r_{\sigma(k)} = x^*$ so by Lemma 3.1 we have $f_{\sigma(i)}(x^*) \leq x^*$.  Similarly when $i > k$ we have $r_{\sigma(i)} \geq r_{\sigma(k)} = x^*$ so by Lemma 3.1 we have $f_{\sigma(i)}(x^*) \geq x^*$.  Thus we have:
\[ f_{\sigma(1)}(x^*),...,f_{\sigma(k-1)}(x^*) \leq x^* \leq f_{\sigma(k+1)}(x^*),...,f_{\sigma(M)}(x^*) \]
Thus $x_{M+1} = x^*$ so  $x^*$ is a fixed point of the recurrence and therefore it is the unique limit for any initial condition.
\end{proof}
Theorem \ref{periodone} gives the complete solution to the rank-type equation \eqref{rankeq1} in the period one case.  We now turn to the case of period two forcing and restrict our attention to the max-type equation \eqref{maxeq}, for which it is possible to find a closed-form solution.    This solution gives insight into the complexity of solutions to \eqref{maxeq} when the forcing period is large.  
For $M=P=2$,  the equation (\ref{rankeq1}) becomes the period two recurrence
\begin{eqnarray}\label{tte1}
 x_{2i} &=& \max \{ f_1(x_{2i-1}) , f_2(x_{2i-2}) \} \nonumber \\
 x_{2i+1} &=& \max \{ g_1(x_{2i}) , g_2(x_{2i-1}) \} 
\end{eqnarray}
where $f_1, f_2, g_1$ and $g_2$ are contractions.  
%
%Since Theorem \ref{contractiontheorem} proves that $f$ is a contraction, it 
%must have a unique fixed point, and so we only need to show that $r_f$ is indeed a fixed point of $f$.  
%This strategy of first showing the $f$ was a contraction and then simply demonstrating the fixed point can be easily generalized to our period two recurrence.  
We can denote the fixed points
\begin{eqnarray*}
 f_1(g_1(r_1)) &=& f_1(r_2) = r_1 \\
 g_1(f_1(r_2)) &=& g_1(r_1) =  r_2 \\
  f_2(r_3) &=& r_3 \\
  g_2(r_4) &=& r_4. 
 \end{eqnarray*}
 
\begin{theorem} \label{explicit} The period-two orbit
\begin{eqnarray} \label{maxsol}
x_{2i} &=& \max\{f_1(\max\{r_2,r_4\}),r_3\}\nonumber \\
x_{2i+1} &=& \max\{g_1(\max\{r_1,r_3\}),r_4\}
\end{eqnarray}
for $i\ge 0$ is the attracting period-two orbit of difference equation (\ref{tte1}).
\end{theorem}

\begin{proof} Due to Corollary \ref{rankResult}, it suffices to show that the formula (\ref{maxsol}) gives a period-two orbit. First note the following table can be obtained easily:
\begin{center}
  \begin{tabular}{ | c | c || c | c |}
    \hline
    $\max\{r_1,r_3\}$ & $\max\{r_2,r_4\}$ & $x_{2i}$ & $x_{2i+1}$ \\ \hline\hline
    $r_1$ & $r_2$ & 					$r_1$ & $r_2$ \\ \hline
    $r_3$ & $r_4$ & 					$\max\{r_3,f_1(r_4)\}$ &  $\max\{r_4,g_1(r_3)\}$ \\ \hline
    $r_1$ & $r_4$ & 					$\max\{r_3,f_1(r_4)\}$ & $r_4$ \\ \hline
    $r_3$ & $r_2$ & 					$r_3$ & $\max\{r_4,g_1(r_3)\}$ \\
    \hline
  \end{tabular}
\end{center}
For example, in the first row we have:
\[ x_{2i} = \max\{f_1(\max\{r_2,r_4\}),r_3\}  = \max\{f_1(r_2),r_3\} = \max\{ r_1, r_3 \} = r_1 \]
The rest follow from similar simple logic.  Now, we can make a new table considering all the possibilities for $x_{2i}$, $x_{2i+1}$, and $x_{2i+2} = \max\{f_1(x_{2i+1}),f_2(x_{2i})\}$.
\begin{center}
  \begin{tabular}{ | c | c || c | c || c |}
    \hline
    $x_{2i}$ & $x_{2i+1}$ &				$f_1(x_{2i+1})$ &		$f_2(x_{2i})$ &	$x_{2i+2}$ \\ \hline\hline
    $r_1$ & $r_2$ &				$r_1$ &			$f_2(r_1)$ &	$r_1$ \\ \hline
    $r_3$ & $r_4$ &		 		$f_1(r_4)$ &		$r_3$ &		 $r_3$\\ \hline
    $f_1(r_4)$ & $r_4$ &			$f_1(r_4)$ &		$f_2(f_1(r_4))$& $f_1(r_4)$ \\ \hline
    $r_3$ & $ g_1(r_3)$ &			$f_1(g_1(r_3))$ &	$r_3$&		$r_3$ \\ \hline
  \end{tabular}
\end{center}

 Since columns three and four are clear, we must justify the final column using Lemma \ref{lem1}, (note that the final column is simply the max of columns three and four).  In the first row, since $r_1 > r_3$ and $r_3$ is the fixed point of $f_2$, Lemma \ref{lem1} implies that $f_2(r_1) < r_1$.  In the second row, note that $r_3  = \max\{ f_1(r_4), r_3\}$ so $r_3 > f_1(r_4)$.  In the third row, note that $f_1(r_4)  > r_3$ and since $r_3$ is the fixed point of $f_2$ by Lemma \ref{lem1} we conclude that $f_2(f_1(r_4)) < f_1(r_4)$.  Finally, in the fourth column note that this combination only occurs when $r_3 > r_1$ and $r_1$ is the fixed point of $f_1 \circ g_1$ so by Lemma \ref{lem1} we have $f_1(g_1(r_3)) < r_3$.  Thus the table shows that $x_{2i+2}=x_{2i}$.  

Note that one combination, $x_{2i} = f_1(r_4)$ and $x_{2i+1} = g_1(r_3)$ is missing from the table. If this combination occurred, then $f_1(r_2)= r_1<r_3\leq f_1(r_4)$ and $ g_1(r_1)=r_2 < r_4 \leq g_1(r_3)$, which contradicts Lemma \ref{lem21}.  Thus this combination is impossible, and the four rows of the table represent all possibilities.

We now construct an analogous table to show that $x_{2i+3}=x_{2i+1}$, which will complete the proof.
\begin{center}
  \begin{tabular}{ | c | c || c | c || c |}
    \hline
    $x_{2i+2}$ & $x_{2i+1}$ &				$g_1(x_{2i+2})$ &		$g_2(x_{2i+1})$ &	$x_{2i+3}$ \\ \hline\hline
    $r_1$ & $r_2$ &				$r_2$ &			$g_2(r_2)$ &	$r_2$ \\ \hline
    $r_3$ & $r_4$ &		 		$g_1(r_3)$ &		$r_4$ &		 $r_4$\\ \hline
    $f_1(r_4)$ & $r_4$ &			$g_1(f_1(r_4))$ &		$r_4$& $r_4$ \\ \hline
    $r_3$ & $ g_1(r_3)$ &			$g_1(r_3)$ &	$g_2(g_1(r_3))$&		$g_1(r_3)$ \\ \hline
  \end{tabular}
\end{center}
In the first row, $r_2 > r_4$ so by Lemma \ref{lem1}, $g_2(r_2) < r_2$.  In the second row, $r_4  = \max\{g_1(r_3),r_4\}$ so $r_4 > g_1(r_3)$.  In the third row, $r_4 > r_2$ so by Lemma \ref{lem1} $g_1(f_1(r_4) < r_4$.  In the final row, $g_1(r_3) > r_4$ so by Lemma \ref{lem1} we have $g_2(g_1(r_3)) < g_1(r_3)$.  This completes the justification of this table, and finishes the proof.
\end{proof}
\begin{example}
\label{ex44}
We return to the equation from Example \ref{ex1} with $M = P = 2$:
\begin{eqnarray*} 
x_{2i}&=& \max\{A_{11}x_{2i-1}^{\alpha_1},A_{21}x_{2i-2}^{\alpha_2} \}\\
x_{2i+1} &=& \max\{A_{12}x_{2i}^{\alpha_1},A_{22}x_{2i-1}^{\alpha_2} \}
\end{eqnarray*}
where $A_{jk} > 0$ and $-1<\alpha_1,\alpha_2<1$. We rewrite this equation by taking the natural log of each term (since $\ln$ is monotonic) to get
\begin{eqnarray*} 
y_{2i}&=& \max\{\ln A_{11}+\alpha_1y_{2i-1},\ln A_{21}+\alpha_2 y_{2i-2}\}\\
y_{2i+1} &=& \max\{\ln A_{12}+\alpha_1 y_{2i},\ln A_{22}+\alpha_2 y_{2i-1} \}.
\end{eqnarray*}
In terms of Theorem 4.2 we have:
\begin{eqnarray*}
r_1 &=& \frac{\ln A_{12} + \alpha_1 \ln A_{11}}{1-\alpha_1^2}\\
r_2 &=& \frac{\ln A_{11} + \alpha_1 \ln A_{12}}{1-\alpha_1^2}\\
r_3 &=& \frac{\ln A_{22}}{1-\alpha_2}\\
r_4 &=& \frac{\ln A_{21}}{1-\alpha_2}
\end{eqnarray*}
Then the period two limit is  defined in Theorem 4.2, and we can exponentiate to get the period two limit of $x_n$.  Thus we have
\begin{eqnarray*}
x_{2i} &=& \max \left\{ A_{12}A_{21}^{\frac{\alpha_1}{1-\alpha_2}},A_{12}A_{11}^{\frac{\alpha_1}{1-\alpha_1^2}}A_{12}^{\frac{\alpha_1^2}{1-\alpha_1^2}}, A_{22}^{\frac{1}{1-\alpha_2}} \right\} \\
x_{2i+1}&=& \max \left\{ A_{11}A_{22}^{\frac{\alpha_1}{1-\alpha_2}},A_{11}A_{12}^{\frac{\alpha_1}{1-\alpha_1^2}}A_{11}^{\frac{\alpha_1^2}{1-\alpha_1^2}}, A_{21}^{\frac{1}{1-\alpha_2}} \right\} 
\end{eqnarray*}
as the periodic limit.  This solution is consistent with the autonomous case  in Example 2.6 of \cite{sauerMax}. Setting $A_1=A_{11}=A_{12}$ and $A_2=A_{21}=A_{22}$, the solution is found to be simply $x_n \to \max\left\{ A_1^{\frac{1}{1-\alpha_1}}, A_2^{\frac{1}{1-\alpha_2}} \right\}$.
\end{example} 
\end{section}

\begin{section}{Discussion}
We have shown in Corollary \ref{rankResult} that solutions of periodically-forced rank-type difference equations are asymptotically periodic of the forcing period. The same is true of a class of more general equations called sup-contractive, according to Corollary \ref{mainresult}. In some simple cases, we were able to identify explicit solutions, as in Theorem \ref{periodone} and Theorem \ref{explicit}. The solutions appear to be significantly more complicated for larger period $P$ and memory $M$ than treated here, and explicit formulas for the solutions remain to be found.

\bigskip

{\bf Acknowledgments.} This research was partially supported by NSF under DMS-0508175 and DMS-0811096.
%
%The nonexpansive case is a much more difficult problem.  For example, even in the autonomous case $P=1$, it is known \cite{GL} that difference equations of type 
%\[ x_n = \textup{$k$-rank}\{-x_{n-1} + b_1,...,-x_{n-M} + b_M \}, \]
%referred to as {\em nonhyperbolic} in \cite{sauerRank}, may have solutions whose prime period is  greater than $P$.  
\end{section}

\end{document}